\pdfoutput=1
\documentclass[11pt]{amsart}
\usepackage{pb-diagram}
\usepackage{amsmath}
\usepackage{tikz}

\usepackage{pgf}
\def\obrazek#1#2{\pgfdeclareimage[height=#2]{#1}{#1}
\begin{center}\pgfuseimage{#1}\end{center}}

\usetikzlibrary{matrix,arrows,decorations.pathmorphing}

\numberwithin{equation}{section}
\newtheorem{thm}[equation]{Theorem}

\newtheorem{prp}[equation]{Proposition}

\newtheorem{cor}[equation]{Corollary}
\theoremstyle{remark}
\newtheorem{rem}[equation]{Remark}
\newtheorem{exa}[equation]{Example}
\newtheorem{df}[equation]{Definition}

\DeclareMathOperator{\map}{{\bf map}}
\DeclareMathOperator{\en}{{\bf end}}
\def\sec{{\bf sec}}

\def\R{\mathbb{R}}
\def\Z{\mathbb{Z}}
\def\N{\mathbb{N}}
\def\Q{\mathbb{Q}}
\def\F{\mathbb{F}}

\def\fU{\mathfrak{U}}
\def\fF{\mathfrak{F}}
\def\vP{\vec{P}}
\def\vSigma{\vec{\Sigma}}
\def\vR{\vec{\R}}
\def\vOmega{\vec{\Omega}}
\def\vxOmega{\vec{\Omega}\vSigma X}
\def\vixOmega{\vec{\Omega}_{inc}\vSigma X}
\def\vOSX{\vec{\Omega}\Sigma X}
\def\SX{\Sigma X}
\def\H{\widetilde H}

\def\l{[\!(}
\def\r{)\!]}

\def\Ps{\vec{P}}
\def\Ss{\vec{S}}

\title[\today]{Paths of the directed suspension}
\author{ Andrzej Weber \and Krzysztof Ziemia\'nski}
\address{University of Warsaw, Institute of Mathematics \\
 Banacha 2, 02-097 Warszawa, Poland\\
and
 Institute of Mathematics\\ Polish Academy of Sciences\\ ul. \'Sniadeckich 8, 00-656 Warszawa, Poland}
 \email{aweber@mimuw.edu.pl}

\address{University of Warsaw, Institute of Mathematics\\
 Banacha 2, 02-097 Warszawa, Poland}
 \email{ziemians@mimuw.edu.pl}

\begin{document}

\begin{abstract}
	We prove that the loop space of the directed suspension of a directed space is homotopy equivalent to the James construction. In particular, it does not depend on the directed structure of a given directed space.
\end{abstract}

\maketitle

\section{Introduction}
The operation of suspension of a topological space
$$X\mapsto \Sigma X=S^1\wedge X$$
is far from being faithful. Whole information about noncommutativity of $\pi_1(X,x_0)$ is lost. One can read only the abelianization
$$\pi_1(X,x_0)_{ab}=H_2(\Sigma X)\,.$$
Similarly, the product structure in the cohomology $H^*(X)$ is lost. Multiplication in $H^{>0}(\Sigma X)$ is trivial. Recently a new kind of structures came to attention of topologists. One considers a subset of \emph{allowable} paths $\Ps X\subset PX$ in the space of all paths in $X$. This kind of structure is called a directed structure, or a d-stucture \cite{Gr}. Any topological space $X$ has many directed structures. The poorest one contains constant paths only; this is the discrete structure $X_\delta$. Another structure is the richest: all paths are allowed. This is the total structure $X_{tot}$. Of course, there are many intermediate structures. The purpose of our paper is to show that, in some sense, the suspension forgets information about the directed structure.

 Given a topological space $X$ with a chosen directed structure, there are many ways of constructing directed structure on $\Sigma X$. There are two particularly interesting cases to consider. One is the smash product with the circle $S^1$ with total structure
$$\Sigma X=S^1_{tot}\wedge X\,,$$
the other one is the smash product with the oriented circle $S^1$, in which only the paths with non-decreasing angle are allowed. This suspension is denoted by
$$\vSigma X=\Ss^1\wedge X\,.$$
In both cases, we can consider either the minimal structure, which contains only finite concatenations of directed paths coming from $S^1_{tot}\times X$ or $\Ss^1\times X$, or the complete one, which allows infinite concatenations. Obviously, all these directed structures depend on the directed structure in $X$. We formulate our results in a way that they can be applied to a wider variety of directed structures. We will not assume explicitly that the directed structure on $\Sigma X$ comes from a directed structure on $X$, only that it satisfies certain conditions.
We will study  the spaces of directed loops $\vOmega\Sigma  X$ based at the distinguished point of the suspension. It is a subspace of the space of all loops $\Omega\Sigma X$.

\begin{thm} \label{twierdzenieglowne}
	Suppose that $X$ is a connected space. Furthermore, assume that the suspension $\Sigma X$ is equipped with a directed structure that is excisive, finitely generated and translatable (see Def \ref{fingen},  \ref{translatable} and \ref{exci}). With the notation as above, the inclusion
\begin{equation}
	\vOmega\Sigma X	\subseteq\Omega\Sigma X\end{equation}
is a weak homotopy equivalence.
\end{thm}

The technical assumptions  of Theorem \ref{twierdzenieglowne} are satisfied for a natural class of directed spaces:

\begin{thm} The conclusion of the Theorem \ref{twierdzenieglowne} holds for the suspension of the realization of cubical complexes with minimal and completed directed structures.\end{thm}

The homotopy type of the space $\Omega\Sigma X$ does not depend on the directed structure of $X$. 
Therefore, the homotopy type of $\vOmega\Sigma X$ do not depend on the directed structure in $\Sigma X$.

The Theorem \ref{twierdzenieglowne} is a consequence of the following statements (Theorem \ref{homomain}, Corollary \ref{cormain} and Theorem \ref{spojnosc}): the natural map
\begin{equation}\label{JavO}
	JX\to \vOmega\Sigma X
\end{equation}
is a weak homotopy equivalence. Here $JX$ is the free topological monoid generated by $X$, which is called the James construction. By the classical result of James \cite{Ja} (see also \cite[\S5.3]{CaMi}) the natural map
\begin{equation}JX\to \Omega\Sigma X\end{equation}
is also a weak homotopy equivalence.   
In another paper \cite{Mi}, Milnor have proved by simplicial methods that $\Omega\Sigma X$ is homotopy equivalent to the topological group $GX$ freely generated by $X$. These result coincide since $JX$ is connected, and then the group completion does not change the homotopy type. If $X$ is not connected then one is tempted to prove that (\ref{JavO})
is a homotopy equivalence. We skip that proof due to point-set topology problems. We rely only on the homological argument which works fine only for connected spaces. 

We note that in general $\vOmega\Sigma X\to \Omega\Sigma X$ is not a homotopy equivalence. For $X=S^0$ we have
$$\vOmega\vSigma X=\vOmega\Ss^1\sim JS^0=\N,$$
$$\Omega\vSigma X=\Omega\Ss^1\sim GS^0=\Z$$
and the inclusion is homotopic to the inclusion of the natural numbers into the integer numbers $\N\hookrightarrow \Z$.

We construct an explicit homotopy inverse of (\ref{JavO}). We note that the original proof of James and other proofs which are  present in the literature for the classical case do not exhibit an inverse map. Our construction of the inverse map is specific for the directed spaces.

In the course of the proof we meet several technical problems. Most of them are caused by the fact  that the suspension $\vSigma X$ might have pretty bad structure in the neighbourhood of the distinguished base point. For any space $X$, the suspension admits the effect of "peacock feather eye", i.e. there are directed loops which infinitely times pass through the distinguished point.

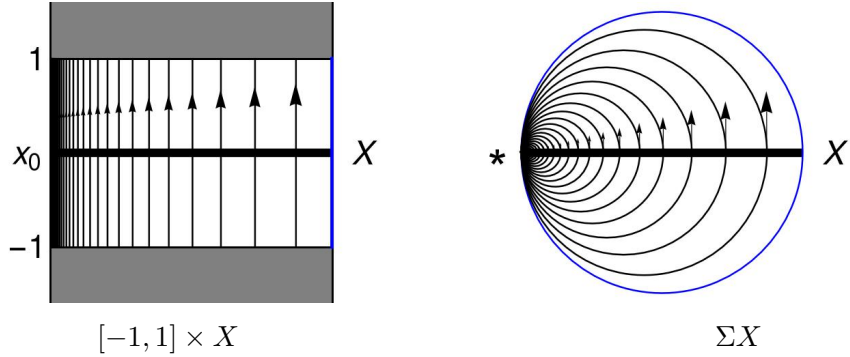
\obrazek{rys2}{4cm}

\begin{center}$[-1,1]\times X$\hskip180pt $\Sigma X$\end{center}

\noindent To avoid this problem we first consider the space of paths which are finite concatenations of paths coming from $[-1,1]\times X$, i.e.~we consider the minimal directed structure and later pass to the closure.

\medskip
The paper is organized as follows: after introducing the notions in \S\ref{directed:spaces} and \ref{section:cubical} 
we discuss directed structures in the smash product $X\wedge Y$. In \S\ref{smash} we show under suitable conditions that if $X$ and $Y$ are connected directed spaces, then $\vOmega(X\wedge Y)$ is connected. Further in \S\ref{suspension} we describe a convenient model of the suspension $\Sigma X$. The James construction is recalled in \S\ref{James}.
We modify  the classical homological proof of James in the context of directed spaces. That is done in the \S\ref{homologicznie}. In the next section \S\ref{homotopijnie}, we construct a homotopy inverse of the James map (\ref{JavO}) explicitly. We note that such a construction was not done in the classical context and might be much harder for non-directed paths.
Finally, in \S\ref{straightening} we show how to  deform in a canonical way a path to another one which corresponds to a configuration of points in $X$, that is to a path which lies in the image of the James map.

\section{Directed spaces}
\label{directed:spaces}

Let $(X,x_0)$ be a pointed Hausdorff space, which is locally contractible.
The following definition is due to Grandis \cite{Gr}

\begin{df}
	\emph{A d-structure} on a topological space $X$ is a family of paths 
\[
	\mathcal{D}\subseteq P(X):=\map([0,1],X),
\]
	 called \emph{d-paths}, that satisfies the following conditions:
	\begin{itemize}
	\item{Every constant path a d-path.}
	\item{The concatenation of d-paths is a d-path.}
	\item{Non-decreasing reparametrization of a d-path is a d-path.}
	\end{itemize}
	\emph{A d-space}, or \emph{a directed space}, is a space equipped with a d-structure. If it is clear what d-structure is considered on a space $X$, it is denoted by $\vP(X)$.
\end{df} 
A map $f:X\to Y$, where $X$, $Y$ are d-spaces, is \emph{a d-map} if it preserves a d-structure, i.e., if $f(\alpha)\in\vec{P}(Y)$ whenever $\alpha\in\vP(X)$. The d-spaces and d-maps form a category, which is complete and cocomplete \cite{Gr}. For every family $\mathcal{F}\subseteq P(X)$ there exists the unique minimal d-structure $\overline{\mathcal{F}}\supseteq\mathcal{F}$ on $X$ generated by $\mathcal{F}$. It contains all constant paths and all paths having the form
\[
	(\alpha_1*\dots*\alpha_n)\circ f,
\]
where $\alpha_i\in\mathcal{F}$ for $i\in\{1,\dots,n\}$ and $f:[0,1]\to[0,1]$ is a continuous non-decreasing map.

If $X$ is a d-space, and $A,B\subseteq X$, then we denote
\begin{equation}
	\vP(X)_A^B:=\{\alpha\in \vP(X):\; \alpha(0)\in A,\; \alpha(1)\in B\}.
\end{equation}
If $(X,x_0)$ is a pointed d-space, we define \emph{the space of directed loops on $X$} as
\begin{equation}
	\vOmega(X):=\vP(X)_{x_0}^{x_0}.
\end{equation}
A path $\alpha:[a,b]\to X$, for $(a,b)$ not necessarily equal to $(0,1)$, will be called a d-path if its reparametrization $[0,1]\ni t\mapsto \alpha(a+t(b-a))\in X$ is a d-path.

Here we list examples of d-spaces that play an important role in this paper:
\begin{itemize}
	\item{Any topological space has two extreme d-structures: the total $X_{tot}$ one which contains all paths, and the discrete one $X_\delta$ which contains only constant paths.}
	\item{\emph{A directed Euclidean space $\vec{\R}^n$}, where d-paths are all paths having non-decreasing coordinates.}
	\item{Every subspace $Y\subseteq X$ of a d-space  is a d-space, with a d-structure $\vP(X)\cap P(Y)$. Particular cases are $\vec{I}=[0,1]\subseteq \vec{\R}$ (\emph{the directed interval}) and $\vec{I}^n\subseteq \vec{\R}^n$ (\emph{the directed $n$--cube}).}
	\item{For every closed subset $Y\subseteq X$ of a d-space, the quotient $X/Y$ is also a d-space, with a d-structure generated by paths having the form $p\circ\alpha$, where $\alpha\in\vP(X)$ and $p:X\to X/Y$ is the projection. 
We obtain the minimal quotient d-structure $P_{\text{min}}(X/Y)$.}
	\item{The completed quotient d-structure on $X/Y$, denoted by $\vP_c(X/Y)$, that contains all paths $\alpha$ for which there exists a (non-necessarily finite) cover of $[0,1]$ by closed intervals $E_a$, $a\in A$, such that $\alpha|_{E_a}$ is a projection of a d-path on $X$. The completed structure is the completion of the minimal d-structure in the sense of \cite{Z2}.}
\item{An important special case is \emph{the directed circle} 
\[	\vec{S}^1=\vec{I}/\{0,1\}\cong \vec{\R}/\big((-\infty,-1]\cup[1,+\infty)\big).\]
	In this case the minimal and the completed d-structure are equal.
	}
\end{itemize}
Note that the minimal structure in $X/Y$ is not always natural. For example, take the interval $X=[-1,1]$ with the total structure. Then $X/\{-1,1\}\simeq S^1$ as topological spaces. But the directed paths $\alpha:[a,b]\to X/\{-1,1\}$ of the minimal quotient structure have the property that $\alpha^{-1}(*)$ has only finitely many connected components. 

\begin{df}
	Let $X$ be a space and let $\vP(X)$ be a d-structure on $X$. \emph{The completion} of $\vP(X)$ is a d-structure $\vP_c(X)$ on $X$ such that $\alpha\in \vP_c(X)$ if and only if there exists a cover
	\[
		\bigcup_{i\in J}	 [a_i,b_i]=[0,1]
	\]
	such that $\alpha|_{[a_i,.b_i]}\in \vP(X)$ for all $i\in J$.
\end{df}

\section{Cubical complexes}
\label{section:cubical}
The completion of the minimal structure on a quotient space will be called \emph{the completed structure}. Clearly, $\vP_c(X_{tot}/Y_{tot})\cong P_{tot}(X/Y)$.

Semi-cubical complexes \cite{FGR}  form a family of d-spaces, which is especially important because of their applications in Computer Science. In this paper we consider cubical sets: a class which contains semi-cubical sets and is closed with respect to taking suspensions.

\begin{df}[{\cite{K}}]
	\emph{A cubical set} $K$ is a sequence of disjoint sets $(K_n)$, $n=0,1,\dots$ equipped with 
	\begin{itemize}
		\item{face maps $d^\varepsilon_{i}:K_n\to K_{n-1}$, for $n>0$, $i\in\{1,\dots,n\}$, $\varepsilon\in\{0,1\}$,}
		\item{degeneracy maps $s_i:K_n\to K_{n+1}$, for $n\geq 0$, $i=1,\dots,n+1$,}
	\end{itemize}
	which satisfy the following cubical relations:
	\begin{itemize}
		\item{$d^\varepsilon_i d^\eta_j=d^\eta_{j-1} d^\varepsilon_i$ for $i<j$,}
		\item{$s_is_j=s_{j-1} s_i$ for $i<j$,}
		\item{$s_jd^\varepsilon_i=\begin{cases}
			d^\varepsilon_i s_{j-1} &\text{for $i<j$,}\\
			\text{identity} & \text{for $i=j$,}\\
			d^\varepsilon_{i-1}s_j & \text{for $i>j$.}
		\end{cases}$}
	\end{itemize}
\end{df}

Every cubical set has a geometric realization which is a d-space. For $\varepsilon\in\{0,1\}$ and $i\in\{1,\dots,n\}$, define d-maps
\begin{align*}
	\delta^\varepsilon_i: \vec{I}^{n-1} \ni (t_1,\dots,t_{n-1})& \mapsto (t_1,\dots,t_{i-1},\varepsilon,t_{i+1},\dots,t_n)\in \vec{I}^n\\
	\sigma_i:\vec{I}^{n+1}\ni (t_1,\dots,t_{n+1})&\mapsto (t_1,\dots,t_{i-1},t_{i+1},\dots,t_{n+1})\in\vec{I^n}.
\end{align*}	
\emph{The geometric realization} of a cubical set $K$ is the quotient space
\begin{equation}
	|K|=\coprod_{n\geq 0} K_n \times \vec{I}^n/\sim,
\end{equation}
where $\sim$ is generated by $(c,\delta^\varepsilon_i(\mathbf{t}))\sim (d^\varepsilon_i(c),\mathbf{t})$ and $(c,\sigma_j(\mathbf{t}))\sim (s_j(c),\mathbf{t})$.

There are (at least) two natural d-structures on $|K|$:
\begin{itemize}
\item{The minimal quotient d-structure $\vP_\text{min}(|K|)$, namely, a continues path $\alpha$ is a d-path if there exists a sequence of numbers $0=t_0<t_1<\dots<t_k=1$, cubes $c_i\in K_{d(i)}$ and d-paths $\beta_i:[t_{i-1},t_i]\to \vec{I}^n$ such that $\alpha(t)=(c_i,\beta_i(t))$ for $t\in [t_{i-1},t_i]$.
This is a minimal d-structure such that the inclusions of cubes induce d-maps.
}
\item{The completed d-structure $\vP_{c}(|K|)$, namely, a path $\alpha$ is a d-path if there exists a (non-necessarily finite) cover of $[0,1]$ by closed intervals $E_a$, $a\in A$, such that $\alpha|_{E_a}$ has the form $(c,\beta)$, for $c\in K_n$ and a d-path $\beta:E_a\to \vec{I}^n$. This coincides with the completed d-structure on a quotient space.}
\end{itemize}

The geometric realization of $K$ with the minimal quotient (resp. the completed) d-structure completed d-structure will be denoted by $|K|_{\text{min}}$ (resp.\ $|K|_c$). 

We will need the following technical result.

\begin{prp}\label{cubicalNeighborhood}
	Let $K$ be a cubical set, $L$ its subset, and $p\in K_0$. Then there exists a subset $A\subseteq |K|$ such that the interior of $A$ contains $|L|$ and the maps 
\[
	\vP(|K|)_p^{|L|}\to \vP(|K|)_p^A,
\]	
	induced by the inclusion is a homotopy equivalence, where $\vP$ stands for either $\vP_\text{min}$ or $\vP_c$.
\end{prp}
\begin{proof}
	The same argument is valid in both cases. Let $r:\vec{I}\to\vec{I}$ be a d-map given by the formula
	\[
		r(s)=\begin{cases}
			0 & \text{for $s\in[0,\tfrac{1}{3}]$,}\\
			3(s-\tfrac{1}{2})+\tfrac{1}{2} & \text{for $s\in[\tfrac{1}{3},\tfrac{2}{3}]$,}\\
			1 & \text{for $s\in[\tfrac{2}{3},1]$,}
		\end{cases}
	\]
	and let $r^n:\vec{I}^n\to\vec{I}^n$ denote its $n$--th product. Furthermore, let $r^n_t=(1-t)id_{\vec{I}^n}+tr^n$, $t\in[0,1]$, be a homotopy between the identity map on $\vec{I}^n$ and $r^n$. For every $t$, the family $r^n_t$, $n\geq 0$, commutes with the face maps $\delta^\varepsilon_i$ and the degeneracy maps $\sigma_i$; as a consequence, the formula
\[
	R_t:|K|\ni (c,(t_1,\dots,t_n))\mapsto (c,r^n(t_1,\dots,t_n))\in |K|
\]	
	defines a homotopy between the identity map on $|K|$ and $R_1:|K|\to |K|$. Notice that, for every $n\geq 0$, the interior of $(r^n)^{-1}(\partial \vec{I}^n)$ contains $\partial \vec{I}^n$. Therefore, the interior of  $A=(R_1)^{-1}(|L|)$ contains $|L|$. Finally, we need to show that the inclusion $\vP(|K|)_p^{|L|}\subseteq \vP(|K|)_p^{A}$ is a homotopy equivalence. Indeed, $R_1$ induces the map $(R_1)^*:\vP(|K|)_p^{A}\to \vP(|L|)_p^{|L|}$ which is a homotopy inverse of the inclusion, with suitable homotopies induced by $R_t$.
\end{proof}

\section{Connectedness of $\vOmega(X\wedge Y)$}\label{smash}
Here we prove a general fact about  directed structures in the smash product
\[
	X\wedge Y=(X\times Y)/(X\vee Y)
\]
of pointed directed spaces.
We will apply it for $\Sigma X=S^1_{tot}\wedge X$ and $\vSigma X=\vec{S}^1\wedge X$.

\medskip
The minimal directed structure $\vP_{\text{min}}(X\wedge Y)\subset P(X\wedge Y)$  consists of the paths which are concatenations of a finite number of d-paths coming from $X\times Y$. 

\begin{df}\label{fingen}
	We say that a directed structure $\vP(X\wedge Y)$ on the smash product of pointed directed spaces $X,Y$ is \emph{finitely generated} if it contains the minimal structure and, for any pair of points $a,b\in X\wedge Y$, the space $\vP_{\text{min}}(X\wedge Y)_a^b$ is dense in $\vP(X\wedge Y)_a^b$. In particular, $\vOmega_{\text{min}}(X\wedge Y)$ is dense in $\vOmega(X\wedge Y)$.\end{df}
	
\begin{prp}\label{p:CompletedDStructureIsFinitelyGenerated}
	The completed d-structure $\vP_c(X\wedge Y)$ is finitely generated.
\end{prp}
\begin{proof}
	Denote $Z=X\times Y$, $V=X\vee Y$ and fix $\alpha\in \vP_c(Z/V)$. Let $\{U_j\}_{j\in J}$ be a decreasing local basis of $V/Z$ at the point $Z$, i.e., such that $U_j\supseteq U_{j'}$ for $j<j'$. Fix $j\in J$. Consider a cover of $[0,1]$ with maximal connected subsets of $\alpha^{-1}(Z\setminus V)$ and $\alpha^{-1}(U_j)$. By the compactness of $[0,1]$, there exists a finite number of pairwise disjoint intervals $E_i=[a_i,b_i]$, $i=1,\dots,n(j)$ such that $\alpha(a_i)=\alpha(b_i)=V$ and $\alpha(t)\in U_j$ whenever $t\not\in \bigcup E_i$. Define a path
	\[
		\alpha_j(t)=\begin{cases}
			\alpha(t) & \text{for $t\in \bigcup E_i$}\\
			V & \text{otherwise.}
		\end{cases}
	\]
	Now $(\alpha_j)_{j\in J}$ is a sequence of paths of $\vP_\text{min}(Z/V)$ that converges to $\alpha$. We have used here generalized sequences (or nets) and our argument works in the general topological setup.
\end{proof}

\begin{thm}\label{spojnosc} Suppose that $X$ and $Y$ are path connected. We assume that the directed structure in $X\wedge Y$ is finitely generated. Then the space of directed loops $\vOmega(X\wedge Y)$ is connected.
\end{thm}

\begin{proof} It is enough to show that  $\vOmega_{\text{min}}(X\wedge Y)$ is connected since the closure of a connected set is connected. Every path in $\vOmega_{\text{min}}(X\wedge Y)$
can be lifted to $X\times Y$ uniquely away from the base point. Therefore, every path
is a concatenation of a finite number of images of paths 
\[\alpha:[0,a]\to X\times Y\]
such that 
$$\alpha(0),\,\alpha(a)\in X\vee Y\,.$$
We will deform such paths to constant paths. 
Let $x_0$ and $y_0$ be the distinguished points in $X$ and $Y$. There are four possibilities:
\begin{enumerate}
\item ~ $\alpha(0)=(x_0,y)$ and $\alpha(a)=(x,y_0)$,
\item ~ $\alpha(0)=(x_0,y)$ and $\alpha(a)=(x_0,y')$,
\item ~ $\alpha(0)=(x,y_0)$ and $\alpha(a)=(x_0,y)$,
\item ~ $\alpha(0)=(x,y_0)$ and $\alpha(a)=(x',y_0)$.
\end{enumerate}
Let $(\alpha_X,\alpha_Y)$ be the coordinates of $\alpha$.
The path $\alpha$ is homotopic to the concatenation
\begin{equation}\label{rozklad}(\alpha_X(0),\alpha_Y)*(\alpha_X,\alpha_Y(a))\,.\end{equation}
Here $ \alpha_X(0)$ and $\alpha_Y(a)$ denote constant paths of the lengths $a$.

In the first case ($\alpha_X(0)=x_0$, $\alpha_Y(a)=y_0$)
such concatenation is projected to a constant path.

In the second case ($\alpha_X(0)=x_0$, $\alpha_X(a)=x_0$), we note that the first component of (\ref{rozklad}) is projected to a constant path. Since the $Y$ is path connected, the second component can be deformed to the path of the form $(\alpha_X,y_0)$. Again, this path is projected to the constant path.

The third and fourth cases are analogous.\end{proof}

\section{The directed  suspension}
\label{suspension}

Let $X$ be a space with a distinguished point $x_0$. We will describe the suspension $\Sigma{X}=S^1\wedge X$ as a quotient of $\R\times X$. First, we present the circle $S^1$ as the quotient space
\[
	S^1=\R/\big((-\infty,-1]\cup [1,\infty)\big)\, .
\]
Then
$\Sigma X$  is identified with the quotient of $\R\times X$ by the subspace 
\[
	\big((-\infty,-1]\cup [1,\infty)\big)\times X\cup \vR\times \{x_0\}\,.
\]
 The point of $\Sigma X$ which is the image of $(t,x)\in \R\times X$ will be denoted by $\l t,x \r$. The distinguished point of $\Sigma X$ is 
 \[
 	*=\l t,x_0\r=\l s,x\r\quad\text{for }|s|\geq 1\,.
 \]
The coordinates in $\Sigma X$ are defined only away from the distinguished point.
Let $$p:\Sigma X\setminus\{*\}\to X$$ and $$h:\Sigma X\setminus\{*\}\to (-1,1)$$ be the projections onto the respective coordinates.

\bigskip
For a given d-structure $\vP(X)$ on $X$, there is a variety of ``natural'' d-structures on $\Sigma{X}$ which can be defined. They are:
\begin{itemize}
	\item{The directed suspension structure $\vP_{\text{min}}(\vec{S}^1\wedge X)$, in which a path is directed if and only if it is a concatenation of a finite number of paths which are the images of  directed paths in $\vR\times X$. It is a minimal structure such that the quotient map $\vec{S}^1\times X\to \Sigma X$ is a d-map.}
	\item{The completed directed suspension structure $\vP_c(\vec{S}^1\wedge X)$, which allows "infinite" concatenations of d-maps.}
	\item{The completed structure $\vP_c({S}^1_{\mathrm{tot}}\wedge X)$, which is the total d-structure on $\Sigma X$.   }
\end{itemize}

We will consider all these cases simultaneously. We will even work in a greater generality and assume only that the d-structure satisfies the following definition:

\begin{df}\label{translatable} We say that a directed structure on $\Sigma X$ is \emph{translatable} if it contains paths
\[
	[-1,1]\ni t\mapsto \l t,x \r\in \Sigma X
\]
for all $x\in X$, and the maps
\[
	\Sigma X\ni \l t,x \r \mapsto \l \lambda t+\mu, x \r\in \Sigma X
\]
preserve the directed structure for all $\lambda,\mu\in\R$, $\lambda>0$, $-\lambda+\mu\leq-1$ and $\lambda+\mu\geq 1$.
\end{df}

All d-structures listed above are translatable. Notice that we do not assume that a translatable d-structure is induced, in any way, from a d-structure on $X$.
Without specifying any particular directed structure in the suspension, we prove the results in Section \ref{homologicznie} for $\vP_{\text{min}}(\vec{S}^1\wedge X)$, $\vP_c(\vec{S}^1\wedge X)$ and $\vP_c({S}^1_{\mathrm{tot}}\wedge X)=P(\Sigma X)$ simultaneously. We will assume that the directed structure is finitely generated (see Definition \ref{fingen}); this is true in the three listed cases thanks to Proposition \ref{p:CompletedDStructureIsFinitelyGenerated}. We will also need  another technical but natural condition, which excludes point-set pathology (see Definition \ref{exci}). 

In Section \ref{homotopijnie} we assume that all directed paths are non-decreasing along $\R$ coordinate away from the base point. This assumption allows to construct a map from the loop space to the James construction.

\section{James construction}
\label{James}

\subsection{One-sided Moore paths}
In this section we will use one-sided Moore paths, i.e., we do not assume that the paths are parametrized by the unit interval, but by an arbitrary interval $[0,e]$ for $e\geq 0$. The concatenation of such paths is strictly associative (when it is defined).

\begin{df}[One-sided Moore paths]\label{Moore-one} Let $Y$ be a directed space, $y_0\in Y$.
By $\Ps(Y,y_0)$ we denote the space of directed Moore paths $\alpha$, such that $\alpha(0)=y_0$.  Formally,
 \begin{multline*}\Ps(Y,y_0)=\{(\alpha,t^\infty_\alpha)\in \map(\R_+,Y)\times \R_+\,:\\ \,\alpha|_{[0,e]}\text{ is directed},\;\alpha(0)=y_0\,,\;\alpha(s)=\alpha(t^\infty_\alpha)\;\text{for }s>e\,\}\,.\end{multline*}
The element of $\Ps(Y,y_0)$ will be regarded as a map  $$\alpha:[0,t^\infty_\alpha]\to Y\,.$$
\end{df}

\begin{prp} The space $\Ps(Y,y_0)$ is contractible.\end{prp}

\begin{proof} The contracting homotopy is given by the formula $H_s(\alpha)=\alpha|_{[0,s\,t_\infty^\alpha]}$ for $s\in [0,1]$.
\end{proof}

We will use the following notation:
 $$\en:\Ps(Y,y_0)\to  Y$$
is the evaluation at the end of the path $\alpha(t^\alpha_\infty)$.

\begin{rem} If $Y$ has a total directed structure, i.e., every path is directed, then the map "$\en$" is a Serre fibration with the fiber $\Omega(Y,y_0)$. It is not true in general; it might happen that the set $\en^{-1}(y)$ (paths from $y_0$ to $y$) is empty.
\end{rem}

\subsection{Paths in the suspension}\label{stozkiL}
Fix a space $X$ and fix a translatable d-structure on $\Sigma X$. Denote by $C_+$ (resp.\ $C_-$) the upper (resp. the lower) sub-cone of $\Sigma X$, i.e.,  the image of $(-\infty,0]\times X$ (resp.\ $[0,\infty)\times X$) in $\Sigma X$.
Let
\begin{align*}&L_-=\en^{-1}(C_-)=\{\alpha\in \Ps(\Sigma X,*)\;:\;\alpha(t_\infty^\alpha)\in C_-\,\}\,,\\
&L_+=\en^{-1}(C_+)=\{\alpha\in \Ps(\Sigma X,*)\;:\;\alpha(t_\infty^\alpha)\in C_+\,\}\,,\\
&L_0=L_+\cap L_-=\en^{-1}(\{0\}\times X)=\{\alpha\in \Ps(\Sigma X,*)\;:\;\alpha(t_\infty^\alpha)\in X\times\{0\}\,\}\,.\end{align*}

For $x\in X$, $a,b\in\R$, $a<b$
let
 $$\beta_x^{a,b}:[0,b-a]\to \Sigma X$$
be the path given by
$$\beta_x^{a,b}(t)=\l a+t,x\r \,.$$
Note that, according to the Definition \ref{translatable}, such paths are directed. 
In general, we will use the letter $\beta$ to denote various versions of paths, which have $x$ fixed. Let us write $\beta(x)$ for the loop $\beta_x^{-1,1}$ and let us consider the map
$$\beta:X\to\vOmega\Sigma X\,,$$
$$x\mapsto \beta(x)=\beta_x^{-1,1}\,.$$

\subsection{A map from $J(X)$ to $\vOSX$}
Let us recall some details from \cite{CaMi}. There a convenient way of constructing a map $J(X)\to \vOSX$ is presented. The map
$\beta:X\to\vOSX$ does not preserve the base point, since the path $\beta(x_0)$ has  length 2.
Let us modify the space $X$. Let 
\[
	X'=X\vee [0,1]=X\sqcup [0,1]/{x_0\sim 1}
\]
and extend the map $\beta$
\[
\beta'(x)=\begin{cases}\beta(x)&\text{if }x\in X\\
\text{constant path of length }2t&\text{if }x=t\in[0,1]\end{cases}
\]
Let $0\in X'$ be the distinguished point in $X'$. The map $\beta'$ preserves the distinguished points; therefore, it induces a map of monoids:
$$J(\beta'):J(X')\to \vOSX\,.$$
Also we have a map
$$J(r):J(X')\to J(X)$$
induced by the retraction $r:X'\to X$. This map is a homotopy equivalence. Therefore, we obtain a map defined up to homotopy
$$J(\beta')\circ J(r)^{-1}:J(X)\to \vOSX\,.$$
Note that the James construction $J(X')$ contains the disjoint union $$\bigsqcup_{n>0} X^n$$
and the restriction of the map $J(\beta')$ to $X^n$ is the concatenation
$$\beta(x_1)*\beta(x_2)*\dots*\beta(x_n)\,.$$
Later in \S\ref{homotopijnie}, we will construct a map $$\sec:\vOSX\to J(X)$$
such that $$\sec\circ J(\beta')(x_1,x_2,\dots,x_n)=[(x_1,x_2,\dots,x_n)]\in J(X)\,.$$
It will follow that $\sec$ is a homotopy inverse of $J(\beta')\circ J(r)^{-1}$.

\section{Main homological result}
\label{homologicznie}

\subsection{The statement of the homological theorem}

In this section, we consider an arbitrary directed structure in $\Sigma X$ which is translatable (\ref{translatable}). Consider the space of directed loops
\[
	\vOSX=\en^{-1}(*)\,.
\]
The homology $H_*(\vOSX)$ with coefficients in a field has an algebra structure since  $\vOSX$ is a topological monoid. All homology groups appearing in this section will have coefficients in a field $\F$. The main result is the following Theorem \ref{homomain} and its corollaries.
We will use singular homology, which works in a general topological context but we need some general condition allowing to apply Mayer-Vietoris sequence  (see \cite[Ch 4.6, p.188]{Sp} for the definition of excisive couple). This technical issue forces a condition which we impose in the directed structure of $\Sigma X$.

\begin{df}\label{exci}
	A d-structure on $\Sigma X$ is \emph{excisive} if the pair $(L_+,L_-)$ (defined in \S\ref{stozkiL}) is an excisive couple in $\vP(\Sigma{X})_*$ in the sense of \cite[Ch 4.6, p.188]{Sp}.
\end{df}

\begin{thm}\label{homomain} Let $X$ be a space, and let $\vP(\Sigma X)$ be a directed structure on $\Sigma X$ that is excisive and translatable. Assume that $\vOSX$ is connected. Let $V=H_{>0}(X)=\H_*(X)$. Then the  map
 $$\beta_*:V\to H_*(\vOSX)$$
induces an isomorphism of algebras
$$T(V)\to H_*(\vOSX)\,,$$
where $T(V)$ is the free algebra generated by the vector space $V$. 
\end{thm}

\begin{cor} \label{cormain} Under assumptions as above, the extension of the inclusion
\[
	\beta:X\hookrightarrow\vOSX
\]
to the James construction
\[
	J(\beta'):J(X')\hookrightarrow\vOSX
\]
is a weak homotopy equivalence.
In particular, the homotopy type of $\vOSX$ does not depend on the directed structure on $\Sigma X$.
\end{cor}

It follows that:
\begin{cor}Under the assumptions as above, the  inclusion
$$\vOmega\Sigma X\to \Omega\Sigma X$$
is a homotopy equivalence.\end{cor}


Finally, we will show that the assumptions on the d-structure on $\Sigma{X}$ are satisfied for directed suspensions of cubical complexes.

\begin{prp}
	Assume that $X$ is a geometric realization of a cubical set $B$, and the directed structure is the directed suspension structure (minimal or completed). Then $\vP(\vSigma X)$ is both translatable and excisive. If $X$ is connected, then $\vOmega \vSigma X$ is connected.
\end{prp}
\begin{proof}
	As before, the translatability is clear. Let $E_-$, $E_+$ be $1$--dimensional cubes with the single non-degenerate vertices $i_-$ and $i_+$ respectively, and let $E=E_{-}\cup_{d^1_1(i_-)=d^0_1(i_+)} E_+$. There is a d-homeomorphism of triples $(|E|,|E_-|,|E_+|)\cong ([-1,1],[-1,0],[0,1])$. As a consequence, the geometric realization of 
\[
	K=E\otimes B/ (E\otimes\{b_0\} \cup \{v_-,v_+\}\otimes B)
\]
is d-homeomorphic to $\Sigma{X}$. Let $K_-$ and $K_+$ be the images of $E_-\otimes B$ and $E_+\otimes B$. There is a d-homeomorphism of triples
\[
	(\Sigma X,C_+,C_-)\cong (|K|,|K_+|,|K_-|).
\]
By \ref{cubicalNeighborhood}, there are subsets $A_-,A_+\subseteq |K|$ that contain $|K_-|$ and $|K_+|$ in their interiors, respectively, and
\[
	L_-\cong \en^{-1}(|K_-|)\subseteq \en^{-1}(A_-),\quad L_+\cong \en^{-1}(|K_+|)\subseteq \en^{-1}(A_+)
\] 
are homotopy equivalences. The pair $(\en^{-1}(A_-),\en^{-1}(A_+))$ is excisive; this implies that $(L_-,L_+)$ is also excisive.

The last essertion follows from Theorem \ref{spojnosc}.
\end{proof}

\subsection{Basic homotopy equivalences}
Fix a d-structure on $\Sigma X$ that is translatable. Let $$\phi^-:\Sigma X\to \Sigma X$$
be the map induced defined by
$$\l s,x\r\mapsto \l 2s-1,x\r\,.$$
This map shrinks $C_-$ to the distinguished point.
The map $\phi^-$ is d-homotopic (i.e., homotopic via d-maps)  to the identity: the homotopy is defined for $t\in[0,1]$:
\begin{equation}\label{phit}\phi^-_t(\l s,x\r)=\l(t+1)s-t,x\r\,.\end{equation}
The map $\phi^-_t$ shrinks the image of $(-\infty,\tfrac{t-1}{t+1}]\times X$ to the distinguished point.

The map $\phi^-$ induces a map of the space of the directed paths. For $\alpha\in L_-$ we have $\phi^-\circ\alpha\in \vOSX$. Therefore

\begin{prp}\label{prop1}The inclusion $i_-:\vOSX\to L_-$ is a homotopy equivalence. Its inverse is given by
$$\Phi^-(\alpha)=\phi^-\circ \alpha\,.\qed$$\end{prp}

Similarly we have a map
 $$\phi^+=\phi^+_1:\Sigma X\to \Sigma X$$
$$\phi^+(\l s,x\r)= \l 2s+1,x\r\,,$$
which shrinks $C_+$ to the distinguished point. The homotopy between $\phi^+$ and the identity is given by the formula
$$\phi^+_t(\l s,x\r)= \l(t+1)s+t,x\r\,.$$
The map $\phi^+_t$ shrinks the image of $[\tfrac{1-t}{t+1},\infty)\times X$ to the distinguished point.

We obtain

\begin{prp}\label{prop2}The inclusion $i_+:\vOSX\to L_+$ is a homotopy equivalence. The inverse is given by
$$\Phi^+(\alpha)=\phi^+\circ \alpha\,.$$\end{prp}

\begin{prp}\label{prop3}The map $F:X\times \vOSX\to L_0$  given by
$$F(x,\alpha)=\alpha*\beta_x^{-1,0}\,.$$ is a homotopy equivalence.
The inverse is given by
$$G(\alpha)=(\en(\alpha),\phi^-\circ \alpha)\in (\{0\}\times X)\times \vOSX\simeq X\times \vOSX \,.$$
\end{prp}

\begin{proof} The composition $$GF((x,\alpha))=G(\alpha*\beta_x^{0,1})=
(x,\phi^-\circ(\alpha*\beta_x^{-1,0}))$$ is joined with the identity by the homotopy
$$H_t((x,\alpha))=\big(x,\phi^-_t\circ(\alpha*\beta_
x^{-1,(t-1)/(1+t)})\big)\,,$$
$$H_0(x,\alpha)=(x,\alpha)\,,\quad H_1(x,\alpha)=GF(x,\alpha)\,,$$
where $\phi^-_t$ is given by (\ref{phit}).

The opposite composition is equal to
$$FG(\alpha)=F(p(\alpha),\phi^-\circ \alpha)=(\phi^-\circ\alpha)*\beta_x^{-1,0}\,.$$
The homotopy to the identity is given by
$$H'_t(\alpha)=(\phi^-_t\circ\alpha)*\beta_x^{(t-1)/(1+t),0}\,,$$
$$H'_0(\alpha)=\alpha\,,\quad H'_1(\alpha)=FG(\alpha)\,.\qedhere$$
\end{proof}

\subsection{Proof of the main homological theorem}

Proof of Theorem \ref{homomain} is a modification of the proof  in the classical situation
\cite{Ja} and \cite[\S5.3]{CaMi}.\medskip

\begin{proof}  We consider the decomposition
$$\Ps(\vOSX,*) =L_-\cup_{L_0}L_+\,.$$ 
We may apply Mayer-Vietoris exact sequence since we know, by assumption, that $(L_-,L_+)$ is an excisive couple. 
Let $j_\pm:L_\pm\to\Ps(\Sigma X,*)$ be the inclusion. Since the space $\Ps(\SX,*)$ is contractible, the Mayer-Vietoris exact sequence for the reduced homology  gives us an isomorphism
$$\H_*(L_0){\xrightarrow{\;(j_-,j_+)\;}} \H_*(L_-)\oplus \H_*(L_+)\,.$$
Due to Propositions \ref{prop1}-\ref{prop3}, we obtain an isomorphism
$$\H_*(X\times \vOSX)\stackrel\simeq\longrightarrow \H_*(\vOSX)\oplus \H_*(\vOSX)\,.$$
The key point is to analyze the isomorphism map in terms of the algebra structure of $H_*(\vOSX)$.
We have a commutative diagram
\begin{equation}\label{dia}
	\begin{diagram}
		\node{\H_*(L_0)}
			\arrow{e,t}{(j_-,j_+)_*}
		\node{\H_*(L_-)\oplus \H_*(L_+)}
			\arrow{s,r}{(\Phi_-,\Phi_+)_*}
	\\
		\node{\H_*(X\times \vOSX)}
			\arrow{e,t}{\simeq}
			\arrow{n,l}{F_*}
		\node{\H_*(\vOSX)\oplus \H_*(\vOSX)}
	\end{diagram}	
\end{equation}

The map
$$\Phi_-\circ j_-\circ F:X\times \vOSX\to \vOSX$$
$$\alpha\mapsto \phi_-\circ(\alpha*\beta_x^{-1,0})=(\phi_-\circ\alpha)*(\phi_-\circ \beta_x^{-1,0})=(\phi_-\circ\alpha)*const$$
is homotopic to the projection on $\vOSX$. On the other hand
$$\Phi_+\circ j_+\circ F:X\times \vOSX\to \vOSX$$
$$\alpha\mapsto \phi_+\circ(\alpha*\beta_x^{-1,0})=(\phi_+\circ\alpha)*(\phi_+\circ \beta_x^{-1,0})=(\phi_+\circ\alpha)*(\phi_+\circ\beta_x^{-1,0})\,.$$
and after a reparametrisation of $\beta$ is equal to
$$(\phi_+\circ\alpha)*\beta_x^{-1,1}=(\phi_+\circ\alpha)*\beta(x)\,.$$
This map is homotopic to
$$\alpha\mapsto \alpha*\beta(x)\,.$$
Denote by $A$ the algebra $H_*(\vOSX)=H_*(\vOSX;\F)$. This is a graded algebra with $A_0=\F$, by the assumption that $\vxOmega$ is connected. Thus,
$$A=\F\oplus\widetilde A\,,\quad \widetilde A=\H_*(\vOSX)=H_{>0}(\vOSX)\,.$$
The reduced homology of $X\times \vOSX$ can be decomposed as
$$\H_*(X\times \vOSX)=\H_*(X)\oplus \H_*(\vOSX)\oplus\H_*(X)\otimes H_*(\vOSX)=V\oplus \widetilde A\oplus V\otimes \widetilde A\,.$$
The bottom isomorphism of the diagram \ref{dia} is of the form
$$V\oplus \widetilde A\oplus V\otimes \widetilde A\to \widetilde A\oplus \widetilde A\,,$$
$$(v,a,w\otimes b)\mapsto (a,\beta_*(v)+\beta_*(w)\cdot b)\,.$$
We subtract one summand $\widetilde A$ from the above mapping and since $V\oplus V\otimes\widetilde A=V\otimes A$  we obtain an isomorphism
$$V\otimes A\to \widetilde A\,,$$
$$w\otimes b\mapsto \beta_*(w)\cdot b\,.$$
 This property characterizes the tensor algebra $T(V)$, see \cite[Prop.~5.3.2]{CaMi}.\end{proof}

\subsection{Proof of Corollary \ref{cormain}}
The spaces $JX$ and $\vOSX$ are topological monoids, therefore their fundamental groups are commutative. A map between such spaces is a weak homotopy equivalence if and only if it induces an isomorphism of homology groups. It is enough to check that the induced maps are isomorphisms for homologies with coefficients in the fields $\Z_p$ (for any prime $p$) and in $\Q$. By \cite{Ja} the homology of $J(X)$ is the free tensor algebra generated by $V=H^{>0}(X)$. The map 
$$X\longrightarrow JX'\stackrel{J(\beta')}\longrightarrow \vOSX$$
gives rise to the map of homology:
$$V\longrightarrow H_*(JX)\simeq T(V)\longrightarrow H^*(\vOSX)\,.$$
By Theorem \ref{homomain} the last map is an isomorphism.
\qed
\section{Topological construction}
\label{homotopijnie}
\subsection{Increasing paths in the directed suspension}

We will concentrate on \emph{the directed suspension} of a pointed d-space $X$, i.e., we assume that $\vP(\vSigma X)$ is either $\vP_{\mathrm{min}}(\vec{S}^1\wedge X)$ or $\vP_{\mathrm{c}}(\vec{S}^1\wedge X)$. We do not assume that $X$ is connected but it is important that, for every directed path $\alpha\in \vP(\vSigma X)$, the composition $h\circ \alpha$ is non-decreasing on any interval on which it is well-defined. To emphasize this, we will write $\vSigma X$ rather than $\Sigma X$.

\begin{df}
	A path $\alpha\in \vP(\vSigma X)$ is \emph{strictly increasing} if, for every $0\leq s <t \leq 1$ one of the following conditions are satisfied:
	\begin{itemize}
		\item{there exists $s<u<t$ such that $\alpha(u)=*$,}
		\item{the map $(h\circ\alpha)|_{\left(s,t\right)}$ is strictly increasing.}
	\end{itemize}
	Let $\vP_{inc}(\vSigma X)\subseteq \vP(\vSigma X)$ be the space of strictly increasing paths, and let
	\[
		\vixOmega:=\{\alpha\in\vP_{inc}(\vSigma{X}):\;  \alpha(t_0^\alpha)=\alpha(t_\infty^\alpha)=*\}
	\]
	 be the space of strictly increasing loops.
\end{df}

\begin{prp}
	The inclusion $\vixOmega\subseteq \vxOmega$ is a homotopy equivalence.
\end{prp}
\begin{proof}
	For every $0<\varepsilon<1$, the map $R:\vxOmega\to \vixOmega$
	\[
		R(\alpha)(t) = \l (h(\alpha(t))+\varepsilon\bar t\,)(1-\varepsilon)^{-1}   , p(\alpha(t)) \r \,,
	\]
	where $$\bar t=\tfrac{t-t^\alpha_0}{t^\alpha_\infty-t^\alpha_0}$$ is the normalized parameter of $\alpha$, is a homotopy inverse.

\obrazek{rys1}{5cm}
$$h(\alpha(t))\quad\quad\quad
h(\alpha(t))+\varepsilon\bar t\quad\quad\quad
\tfrac{ h(\alpha(t))+\varepsilon\bar t}{1-\varepsilon}$$
\begin{center}Time $\times\;\;\R$--coordinate of $R(\alpha)$\end{center}
\end{proof}

\subsection{Topological James equivalence}

\begin{df}
	For an open neighborhood $U$ of $x_0$, and a sequence $(U_i)_{i=1}^k$ of open subsets of $X$ such that $U_i\cap U=\emptyset$, define the following subset of $J(X)$
	\[
		\fU((U_i),U)=\left\lbrace(x_i)_{i=1}^n:\; \exists_{1\leq m_1<\dots< m_k\leq n}\;(\forall_{j\in\{1,\dots,k\}}\; x_{m_j}\in U_j \;\wedge\; \forall_{i\not\in\{m_1,\dots,m_k\}}\; x_i\in U)  \right\rbrace
	\]
	Let $J_c(X)$ denote the set $J(X)$ with the topology generated by the sets $\fU((U_i),U)$.
\end{df}

\begin{rem}
	The identity map $J(X)\to J_c(X)$ is continuous but its inverse is not in general. Nevertheless, it induces a weak homotopy equivalence since it is homeomorphism on each compact subset.
\end{rem}

Let $\Omega_{\text{min}}:=\vOmega_{\text{min,inc}}(\vSigma X)$ be the space of directed strictly increasing Moore paths on $\vSigma{X}$ that consist of finitely many segments. For every path $\omega\in\Omega_{\text{min}}$, the inverse image
\begin{equation}
	\omega^{-1}(X\setminus\{x_0\}\times \{0\})=\{x^\omega_1<\dots<x^\omega_{k(\omega)}\}
\end{equation}
is a finite set.

\begin{prp}
	The map
 \[\sec:\Omega_{\text{\rm min}}\to J_c(X)\]
 \[ \omega\mapsto (x^\omega_1,\dots,x^\omega_{k(\omega)}) \]
	is continuous.
\end{prp}
\begin{proof}
	Let $(U_i)_{i=1}^k$, $U\ni x_0$ be open subsets of $X$ such that $U_i\cap U$ for all $i$. Fix $\omega\in \Omega_{\text{min}}$ such that $\sec(\omega)\in \fU((U_i),U)$. Then there exists a sequence $t_1<\dots<t_k$ such that $\omega(t_i)\in U_i\times \{0\}$ and $\omega(\R_+\setminus\{t_1,\dots,t_k\})\cap (X\setminus\{x_0\})\times \{0\}=\emptyset$. Furthermore, there exist $t^a_i<t_i<t^b_i\in \R_+$ such that $h(\omega(t))<0$ for $t\in\left[t^a_i,t_i\right[$, $h(\omega(t))>0$ for $t\in\left]t_i,t^b_i\right]$ and $p(\omega(t))\in U_i$ for $i\in [t^a_i,t^b_i]$. Let $\fF(K,V)\subseteq \Omega_{\text{min}}$ denote the space of paths $\alpha$ such that $\alpha(K)\subseteq V$.	
	The set
	\begin{multline*}
		\fF\left(\R_+\setminus \bigcup_{i=1}^k (t^a_i,t^b_i), \{\l s,x\r:\; \text{$s\neq 0$ or $x\in U$}\} \right)
		\cap \bigcap_{i=1}^k \fF(\{t^a_i\},\{\l s,x \r:\; -1<s<0 \} )\cap\\
		 \fF(\{t^b_i\},\{\l s,x \r:\; 0<s<1 \}) \cap
		  \fF([t^a_i,t^b_i],\{\l s,x \r:\; 0<s<1,\; x\in U_i \})
	\end{multline*}
	is an open neighborhood of $\omega$ contained in $\sec^{-1}(\fU((U_i),U))$.
\end{proof}

It is clear from the construction that for
$(x_1,x_2,\dots,x_n)\in X^n\subset J(X')$ we have
$$\sec(J(\beta'(x_1,x_2,\dots,x_n))=[(x_1,x_2,\dots,x_n)]\,,$$
i.e.
$$\sec(J(\beta'(x_1,x_2,\dots,x_n))=J(r)(x_1,x_2,\dots,x_n)\,,$$
where $J(r)$ is the map induced by the retraction $r:X'\to X$.
With the assumption that $\vxOmega$ is connected, it follows from Theorem \ref{homomain} that $\sec$ is a homotopy equivalence.

Below we present a construction of a retraction of $\Omega_{\text{min}}$ to a subspace consisting of the paths which are, up to a reparametrization, in the image of $J(\beta')$.

\begin{exa}
	If $X$ is a discrete d-space, then every directed loop $\omega\in \Omega_{\text{\rm min}}$ is a non-decreasing reparametrization of a concatenation
\[
	\beta_{x_1}*\dots*\beta_{x_n}.
\]	
	
	As a consequence, the map $\sec$ induces a continuous bijection between the space of traces (cf.~\cite[Section 2]{R})
\[\vec{T}(\vec\Sigma X_\delta)_*^*:=\vOmega\vSigma X_\delta / \{\text{non-decreasing reparametrizations}\}\]
 and $J_c(X)$.
\end{exa}

\section{Straightening a path}
\label{straightening}

\subsection{A single segment}
We show a construction which allows to construct a deformation of an arbitrary path to the concatenation of paths of the form $\beta(x)$.

Let $\alpha:[0,a]\to C_-\subset\vSigma X$ be a directed increasing path such that
$$\alpha(0)=*\quad\text{and}\quad
\alpha(a)\in \{0\}\times X\,.$$
In order to construct a deformation, we introduce the map
$$\psi_t^-:C_-\to C_-$$
$$(s,x)\mapsto (s-t,x)\,.$$
Now let us define the homotopy
$$ \widetilde H_t^-(\alpha)=(\psi^-_t\circ\alpha) *\beta^{-t,0}_{p(\alpha(a))}(-/a)\, ,$$
i.e.,
$$ \widetilde H_t^-(\alpha)(s)=\begin{cases}\psi^-_t(\alpha(s))&\text{for } s\in[0,a]\\
\l \tfrac{s-a}a-t,p(\alpha(a)) \r &\text{for } s\in[a,a(1+t)].\end{cases}$$

The path $\widetilde H_t^-(\alpha)$ is of length $(1+t)a$. We rescale the parameter linearly in order to obtain a path of length $a$.
The resulting deformation has the property
$$H_0^-(\alpha)=\alpha\,,$$
and $H_1^-(\alpha)(s)$ is a concatenation of a constant path with $\beta^{-1,0}_{p\alpha(a)}(s/2a)$.

Similarly, we deform paths $\alpha:[0,a]\to C_+\subset\vSigma X$ such that
$$\alpha(0)\in \{0\}\times X\quad\text{and}\quad\alpha(a)=*$$
by the formula
$$\widetilde H_t^+(\alpha)=\beta^{0,t}_{p(\alpha(0))}(-/a)*(\psi_{t}^+\circ\alpha)\,,$$
where $\psi_t^+(s,x)=(s+t,x)$. That is
$$ \widetilde H_t^+(\alpha)(s)=\begin{cases}
\l \tfrac{s}a,p(\alpha(0)) \r &\text{for } s\in[0,at]\\
\psi^+_t(\alpha(s-at))&\text{for } s\in[at,a(1+t)].\\\end{cases}$$
We reparametrize $\widetilde H^+_t(\alpha)$ to have a path of length $a$.
Finally, we construct a deformation of a directed increasing path satisfying
\begin{align}\label{pojedyncza1}&\alpha(0)=\alpha(a)=*\,,\\
\label{pojedyncza2}&\alpha(s)\ne*~\text{ for }~s\in(0,a)\,.\end{align}
Such path has a unique intersection with the set $(\{0\}\times X)\setminus \{*\}$. Let $b\in(0,a)$ be the parameter for which $\alpha(b)\in \{0\}\times X$. Then $\alpha$ can be written as the concatenation
$$\alpha=\alpha^-*\alpha^+,$$
with $\alpha^-$ of length $b$.
We define a deformation
$$ \hat H_t(\alpha)=H^-_t(\alpha^-)*H^+_t(\alpha^+)\,.$$ The line coordinate of $\hat H_1(\alpha)$ is piecewise linear (increasing linearly on $[b/2,b]$ and on $[b,(b+a)/2)$ and constant elsewhere). The path  $ H_1(\alpha)$ can be deformed further (in an uniform way) to obtain a path
$$\beta^{-1,1}_{p(\alpha(b))}(s/a)\,.$$
We have obtained a canonical deformation
$$\alpha(s)\;\leadsto\;\beta^{-1,1}_{p(\alpha(b))}(s/a)\,.$$

\obrazek{rys3}{5cm}
\begin{center}$X\;\times\; \R$\end{center}

\noindent We omit the proof that this procedure is continuous with respect to $\alpha$. The rigorous proof is quite tedious and we leave it to the reader since we do not need it in presence of Theorem \ref{homomain} and Corollary \ref{cormain}.

\subsection{Deformation of a chain}
Every path $\alpha\in \Omega_{\text{min}}$ is a concatenation of a chain
$$const_0*\alpha_1*const_1*\alpha_2*const_2*\dots*\alpha_n*const_n\,,$$
where the paths $\alpha_i$ satisfy (\ref{pojedyncza1}--\ref{pojedyncza2}) and $const_i$ are constant paths. These constant paths can be deformed to paths of length zero in a canonical way and the paths $\alpha_i$ are deformed using the homotopy defined in the precious subsection. The resulting retraction $\alpha\mapsto H_1(\alpha)$ leaves unchanged the paths which are the images $J(\beta')(x_1,x_2,\dots,x_n)$.

\section{Concluding remarks}
Let us summarize the results of the homological and geometric constructions.
\begin{itemize}
\item For $X$ connected, according to Corollary \ref{cormain}, the map
\[
	\beta'\circ J(r)^{-1}:J(X)\to \vOmega\Sigma X
\]
is a weak homotopy equivalence under suitable assumptions on the directed structure, including the minimal or completed structure for cubical complexes.

\item We have constructed a map $\sec:\Omega_{\text{min}}\to JX$ such that 
\[
	\sec\circ\beta'\circ \iota:\bigsqcup_{n>0}X^n\to J(X)
\]
is equal to the natural projection. 
Here $\Omega_{\text{min}}$ is a subset of $\vxOmega$, which is homotopy equivalent.
It follows that the map $\sec$ is essentially a homotopy inverse of $\beta'$. 

\item In addition, we show how to deform an individual path to a path which is in the image of $\beta'$. The deformation is done in a canonical way.
\end{itemize}
We hope that our results will add a new flavour to the old classical James construction.


\end{document}